\documentclass[11pt]{amsart}

\usepackage[margin=1.25in]{geometry}

\usepackage{graphicx}

\usepackage{wasysym}
\usepackage[misc]{ifsym}

\usepackage{amsmath, amssymb, amsthm, graphicx, enumerate, tikz, float, bbm, nicefrac, mathrsfs, amsfonts}
\usepackage{mathtools, comment}
\usepackage{subcaption}
\usepackage[colorlinks]{hyperref}
\hypersetup{citecolor=blue}
\usetikzlibrary{matrix,arrows,decorations.pathmorphing}

\newtheorem{theorem}{Theorem}[section]
\newtheorem{lemma}[theorem]{Lemma}
\newtheorem{corollary}[theorem]{Corollary}

\newcommand{\n}[1]{||#1||}

\theoremstyle{definition}
\newtheorem{definition}[theorem]{Definition}
\newtheorem{example}[theorem]{Example}

\newtheorem*{solution*}{Solution}

\theoremstyle{remark}
\newtheorem{remark}[theorem]{Remark}
\newtheorem{claim}[theorem]{Claim}

\numberwithin{equation}{section}

\newcommand{\C}{\mathbb{C}}

\newcommand{\Z}{\mathbb{Z}}
\newcommand{\U}{\mathscr{U}}
\newcommand{\M}{\mathscr{M}}

\begin{document}

\allowdisplaybreaks

\title[$G$-Invariant Spaces of Bounded Measurable Functions]{Spaces of Bounded Measurable Functions Invariant Under a Group Action}

\author{Samuel A. Hokamp}
\address{Department of Mathematics, Sterling College, Sterling, Kansas, 67579}
\email{samuel.hokamp@sterling.edu}


\date{\today}

\keywords{Spaces of continuous functions, group actions, functional analysis.\\\indent\emph{Corresponding author.} Samuel A. Hokamp \Letter~\href{mailto:samuel.hokamp@sterling.edu}{samuel.hokamp@sterling.edu}. \phone~920-634-7356.}

\begin{abstract}
In this paper we characterize spaces of $L^\infty$-functions on a compact Hausdorff space that are invariant under a transitive and continuous group action. This work generalizes the author's 2021 results, found in \cite{hokamp2021certain}, concerning the specific case of unitarily and M\"obius invariant spaces of $L^\infty$-functions defined on the unit sphere in $\C^n$.
\end{abstract}

\maketitle

\section{Introduction}

To understand this paper's place in the literature, we must first understand the relationships between the papers \cite{NR} (Nagel and Rudin, 1976), \cite{hokamp2021certain} (Hokamp, 2021), and \cite{hokampconti}~(Hokamp). Brief descriptions are given below.

In \cite{NR}, Nagel and Rudin determine the closed unitarily invariant spaces of continuous and $L^p$-functions on the unit sphere of $\mathbb{C}^n$, for $1\leq p<\infty$. That is, there exists a collection ${C}$ of (minimal and invariant) spaces of continuous functions such that each closed unitarily invariant space is the closed direct sum of some subcollection of ${C}$. The same result is not shown for $L^\infty$-functions, since for each $L^\infty$-function $f$, the map $u\mapsto f\circ u$ from the unitary group into the $L^\infty$ -functions need not be continuous under the norm topology.

In \cite{hokamp2021certain}, the author formulates a result for $L^\infty$-functions on the unit sphere that is analogous to the results of Nagel and Rudin when the $L^\infty$-functions are endowed with the weak*-topology. The conclusion of this paper is that the same collection of (minimal and invariant) spaces of continuous functions described in \cite{NR} serves as the ``building blocks'' of the weak*-closed unitarily invariant spaces of $L^\infty$-functions via closures of direct sums of subcollections.

In \cite{hokampconti}, the author generalizes the results of Nagel and Rudin in \cite{NR} by considering spaces of complex continuous and $L^p$-functions, for $1\leq p<\infty$, defined on an arbitrary compact Hausdorff space $X$, on which a compact group $G$ acts continuously and transitively. In particular, when a collection $\mathscr{G}$ of (minimal and invariant) spaces of continuous functions with certain properties exists, this collection plays the same role in constructing the closed spaces of continuous and $L^p$-functions on $X$ as the collection from \cite{NR}.

This paper generalizes the results in \cite{hokamp2021certain} and acts as an analogue to \cite{hokampconti}, in that we explore the case of $L^\infty$-functions defined on $X$ endowed with the weak*-topology. The main result (Theorem~\ref{weak*-closed G inv}) is that when the same collection $\mathscr{G}$ of continuous functions on $X$ from \cite{hokampconti} exists (Definition~\ref{G-col def1}), all weak*-closed invariant spaces of $L^\infty$-functions can be constructed by closing the direct sum of some subcollection of $\mathscr{G}$. 

\section{Preliminaries} 


Let $X$ be a compact Hausdorff space and $C(X)$ the space of continuous complex functions with domain $X$. Let $G$ be a compact group (with Haar measure $m$) that acts continuously and transitively on $X$. When we wish to be explicit, the map $\varphi_\alpha:X\to X$ shall denote the action of $\alpha$ on $X$ for each $\alpha\in G$; otherwise, $\alpha x$ denotes the action of $\alpha\in G$ on $x\in X$.

Let $\mu$ denote the unique regular Borel probability measure on $X$ that is invariant under the action of $G$. Specifically,
\begin{equation}\label{measinv}
    \int_X f\,d\mu=\int_X f\circ\varphi_\alpha\,d\mu,
\end{equation}
for all $f\in C(X)$ and $\alpha\in G$. The existence of such a measure is a result of Andr\'e Weil from \cite{weil1940}, and a construction of $\mu$ can be found in \cite{joys} (Theorem~6.2). Throughout the paper, $\mu$ shall refer to this measure.

The notation $L^p(\mu)$ denotes the usual Lebesgue spaces, for $1\leq p\leq\infty$. 
For $Y\subset C(X)$, the uniform closure of $Y$ is denoted $\overline{Y}$, and for $Y\subset L^p(\mu)$, the norm-closure of $Y$ in $L^p(\mu)$ is denoted $\overline{Y}^{p}$. When $L^\infty(\mu)$ is equipped with the weak*-topology, $\overline{Y}^*$ denotes the weak*-closure of $Y\subset L^\infty(\mu)$.

\begin{remark}\label{weak* and weak Lp closure containment}
For $Y\subset L^\infty(\mu)$ convex and $1\leq p<\infty$, we have $$\overline{Y}^{*}\subset \overline{Y}^{p}\cap L^\infty(\mu).$$
\end{remark}

This follows from the local convexity of $L^p(\mu)$, and from the fact that the weak*-topology on $L^\infty(\mu)$ is \textit{stronger} than the topology which $L^\infty(\mu)$ inherits from each $L^p(\mu)$ endowed with the weak topology.

The following is an easy consequence of \eqref{measinv}:

\begin{remark}\label{switcher-mu}
Let $1\leq p<\infty$ and let $p^\prime$ be its conjugate exponent. Then $$\int_X (f\circ\varphi_\alpha)\cdot g\,d\mu=\int_X f\cdot (g\circ\varphi_{\alpha^{-1}})\,d\mu,$$ for $f\in L^p(\mu)$, $g\in L^{p^\prime}(\mu)$, and $\alpha\in G$.
\end{remark}

The following definition has appeared in several sources, such as \cite{hokampconti}, \cite{Izzo2010}, or \cite{RFT}, but no attribution is given. The last citation is the specific case of the unitary group acting on the unit sphere in $\mathbb{C}^n$.

\begin{definition}
A space of complex functions $Y$ defined on $X$ is \textbf{invariant under $G$} (\textbf{$G$-invariant}) if $f\circ \varphi_\alpha\in Y$ for every $f\in Y$ and every $\alpha\in G$.
\end{definition}

\begin{remark}
Since the action is continuous, $C(X)$ is $G$-invariant. Conversely, if $C(X)$ is $G$-invariant, then each action $\varphi_\alpha$ must be continuous.
\end{remark}

\begin{remark}\label{reminv}
The invariance property \eqref{measinv} means $\mu(\alpha E)=\mu(E)$ for every Borel set $E\subset X$ and every $\alpha\in G$. Consequently, \eqref{measinv} holds for every $L^p$-function, and $L^p(\mu)$ is $G$-invariant for all $1\leq p\leq\infty$.
\end{remark}


Definition~\ref{hokampconti2.3} and Definition~\ref{hokampconti2.4}, stated in \cite{hokampconti}, are generalizations of definitions found in \cite{RFT} related to the unitary group.

\begin{definition}[2.3 \cite{hokampconti}]\label{hokampconti2.3}
If $Y$ is $G$-invariant and $T$ is a linear transformation on $Y$, we say $T$ \textbf{commutes with} $G$ if $$T(f\circ\varphi_\alpha)=(Tf)\circ\varphi_\alpha$$ for every $f\in Y$ and every $\alpha\in G$.
\end{definition}

\begin{definition}[2.4 \cite{hokampconti}]\label{hokampconti2.4}
A space $Y\subset C(X)$ is \textbf{$G$-minimal} if it is $G$-invariant and contains no nontrivial $G$-invariant spaces.
\end{definition}

The remaining definitions all come from \cite{hokampconti}.

\begin{definition}[4.1 \cite{hokampconti}]\label{H(x)spaces}
For each $x\in X$, the space $H(x)$ is the set of all continuous functions that are unchanged by the action of any element of $G$ which stabilizes $x$. That is, $$H(x)=\{f\in C(X):f=f\circ\varphi_\alpha,\text{ for all }\alpha\in G\text{ such that }\alpha x=x\}.$$
\end{definition}

\begin{definition}[4.2 \cite{hokampconti}]\label{G-col def1}
Let $\mathscr{G}$ be a collection of spaces in $C(X)$ with these properties:
\begin{itemize}
    \item[(1)] Each $H\in\mathscr{G}$ is a closed $G$-minimal space.
    \item[(2)] Each pair $H_1$ and $H_2$ in $\mathscr{G}$ is orthogonal (in $L^2(\mu)$): If $f_1\in H_1$ and $f_2\in H_2$, then $$\int_X f_1\bar{f_2}\,d\mu=0.$$
    \item[(3)] $L^2(\mu)$ is the direct sum of the spaces in $\mathscr{G}$.
\end{itemize}
We say $\mathscr{G}$ is a \textbf{$G$-collection} if it also possesses the following property:
\begin{itemize}
    \item[($*$)] $\dim (H\cap H(x))=1$ for each $x\in X$ and each $H\in\mathscr{G}$.
\end{itemize}
\end{definition}




Throughout the paper, $\mathscr{G}$ shall denote a $G$-collection of $C(X)$, indexed by $I$, whose elements are denoted $H_i$, for $i\in I$, and further, \textit{we assume that a $G$-collection exists for $X$.}

\begin{remark}
It should be stressed that we are not implying that a $G$-collection \textit{always} exists for any $X$ and $G$. However, a collection of spaces in $C(X)$ lacking at most only property $(*)$ of Definition~\ref{G-col def1} always exists, as a consequence of the Peter-Weyl theorem from \cite{peterweyl}. This collection is necessarily unique.
\end{remark}

\begin{definition}[4.7 \cite{hokampconti}]
We define $\pi_i$ to be the projection of $L^2(\mu)$ onto $H_i$.
\end{definition}

\begin{remark}
In Theorem~4.5 of \cite{hokampconti}, it is shown that each $\pi_i$ commutes with $G$, and to each $x\in X$ there exists a unique $K_x\in H_i$ such that $$\pi_i f=\int_X f(x)K_x\,d\mu(x),$$ for all $f\in L^2(\mu)$. The domain of $\pi_i$ can then be extended to $L^1(\mu)$ by defining $\pi_i f$ to be the above integral for all $f\in L^1(\mu)$.
\end{remark}

\begin{definition}[4.8 \cite{hokampconti}]
For $\Omega\subset I$, $E_\Omega$ denotes the direct sum of the spaces $H_i$ for $i\in\Omega$.
\end{definition}

\begin{remark}\label{E Omega G-inv}
The $G$-invariance of each $E_\Omega$ is a natural consequence of the definition.
\end{remark}

\begin{remark}\label{f_i=pi_i f}
Definition~\ref{G-col def1} yields that each $f\in L^2(\mu)$ has a unique expansion $f=\sum f_{i}$, with each $f_{i}\in H_i$, which converges unconditionally to $f$ in the $L^2$-norm. Since $\pi_{i}$ is the identity map on $H_i$ and the spaces $H_i$ are pairwise orthogonal, we have $f_{i}=\pi_{i}f$ for $i\in I$. Thus, $$f=\sum\pi_{i}f.$$
\end{remark}

Each $\pi_{i}$ is continuous as the orthogonal projection of $L^2(\mu)$ onto the closed subspace $H_i$. Thus, $\pi_{i}$ annihilates a subset of $L^2(\mu)$ if and only if it annihilates its closure. The following is a consequence of this and Remark~\ref{f_i=pi_i f}:

\begin{remark}\label{E Omega L2 pi i}
For each set $\Omega\subset I$, we have $$\overline{E}_\Omega^{2}=\{f\in L^2(\mu):\pi_{i}f=0\text{ when }i\notin\Omega\}.$$
\end{remark}


Finally, the classical results used in this paper can be found in many texts, with the reference given in each instance being just one such place.




\section{Closures of \texorpdfstring{$G$}{G}-Invariant Sets}\label{Ginv Closures}

In this section, we show that $G$-invariance is preserved by closures in the spaces $C(X)$ and $L^p(\mu)$ for $1\leq p\leq\infty$ (Corollaries \ref{Lp and C G-inv closure} and \ref{Linf G-inv closure}). In particular, $G$ induces classes of isometries on $L^p(\mu)$ and on $C(X)$ (Theorem~\ref{Lalpha bli}), as well as a class of weak*-homeomorphisms on $L^\infty(\mu)$ (Theorem~\ref{weak* cont}).

\begin{theorem}\label{Lalpha bli}
Suppose $\mathscr{X}$ is any of the spaces $C(X)$ or $L^p(\mu)$ for $1\leq p\leq\infty$ and $\alpha\in G$. If $L_\alpha:\mathscr{X}\to\mathscr{X}$ is the map given by $L_\alpha f=f\circ\varphi_\alpha$, then $L_\alpha$ is a bijective linear isometry.
\end{theorem}

\begin{proof}
The bijectivity of each $L_\alpha$ is clear because each has an inverse map $L_{\alpha^{-1}}$. The linearity of each $L_\alpha$ is also clear. Further, the invariance property \eqref{measinv} of $\mu$ yields that each $L_\alpha$ is an isometry on $L^p(\mu)$ (the case for $L^\infty(\mu)$ follows from Remark~\ref{reminv}).

To show the same on $C(X)$, we observe that $$|(L_\alpha f)(x)|=|f(\alpha x)|\leq \n{f}\hspace{.125in}\text{ and }\hspace{.125in}|f(x)|=|(L_\alpha f)(\alpha^{-1}x)|\leq\n{L_\alpha f}$$ for all $x\in X$. These inequalities yield that $\n{L_\alpha f}=\n{f}$.
\end{proof}

\begin{corollary}\label{Lp and C G-inv closure}
Suppose $\mathscr{X}$ is any of the spaces $C(X)$ or $L^p(\mu)$ for $1\leq p\leq\infty$. If $Y\subset\mathscr{X}$ is $G$-invariant, then the closure of $Y$ in $\mathscr{X}$ is $G$-invariant.
\end{corollary}

\begin{theorem}\label{weak* cont}
Let $\alpha\in G$. If $L_\alpha:L^\infty(\mu)\to L^\infty(\mu)$ is the map given by $L_\alpha(f)=f\circ\varphi_\alpha$, then $L_\alpha$ is a weak*-homeomorphism.
\end{theorem}

\begin{proof}
Recall the weak*-topology on $L^\infty(\mu)$ is a weak topology induced by the maps on $L^\infty(\mu)$ of the form $$\Lambda_gf=\int_X fg\,d\mu,$$ for some $g\in L^1(\mu)$. Thus, $L_\alpha$ is continuous with respect to the weak*-topology if and only if $\Lambda_g\circ L_\alpha$ is continuous for all maps $\Lambda_g$.

Fix $g\in L^1(\mu)$. We observe that $$(\Lambda_g\circ L_\alpha)(f)=\Lambda_g(f\circ\varphi_\alpha)=\int_X (f\circ\varphi_\alpha)\cdot g\,d\mu=\int_X f\cdot (g\circ\varphi_{\alpha^{-1}})\,d\mu=\Lambda_{g\circ\varphi_{\alpha^{-1}}}f,$$ for every $f\in L^\infty(\mu)$, by Remark~\ref{switcher-mu}. We conclude $L_\alpha$ is continuous on $L^\infty(\mu)$ with respect to the weak*-topology.

Finally, the map $L_{\alpha^{-1}}:L^\infty(\mu)\to L^\infty(\mu)$ given by $L_{\alpha^{-1}}(f)=f\circ\varphi_{\alpha^{-1}}$ is the inverse of $L_\alpha$. By a similar argument, $L_{\alpha^{-1}}$ is continuous with respect to the weak*-topology, and thus $L_\alpha$ is a weak*-homeomorphism.
\end{proof}

\begin{corollary}\label{Linf G-inv closure}
If $Y\subset L^\infty(\mu)$ is $G$-invariant, then $\overline{Y}^{*}$ is $G$-invariant.
\end{corollary}

\begin{remark}
From Remark~\ref{E Omega G-inv} and Corollary~\ref{Linf G-inv closure}, each $\overline{E}^*_\Omega$ is a weak*-closed $G$-invariant subspace of $L^\infty(\mu)$.
\end{remark}

\section{Characterization of Weak*-Closed \texorpdfstring{$G$}{G}-Invariant Subspaces of \texorpdfstring{$L^\infty(\mu)$}{L(u)}}\label{G inv spaces}


In this section, we state and prove our main result (Theorem~\ref{weak*-closed G inv}), which shows that the spaces $\overline{E}_\Omega^*$ are the \textit{only} weak*-closed $G$-invariant subspaces of $L^\infty(\mu)$.

\begin{theorem}\label{weak*-closed G inv}
If $Y$ is a weak*-closed $G$-invariant subspace of $L^\infty(\mu)$, then $Y=\overline{E}_\Omega^{*}$ for some $\Omega\subset I$.
\end{theorem}

This result is an analogue to Theorem~5.1 of \cite{hokampconti}, which is used in its proof:

\begin{theorem}[5.1 \cite{hokampconti}]\label{main C and Lp result}
Let $\mathscr{X}$ be any of the spaces $C(X)$ or $L^p(\mu)$ for $1\leq p<\infty$. If $Y$ is a closed $G$-invariant subspace of $\mathscr{X}$, then $Y$ is the closure of $E_\Omega$ for some $\Omega\subset I$.
\end{theorem}

The set $\Omega$ from Theorem~\ref{main C and Lp result} is the set $\{i\in I:\pi_{i}Y\neq 0\}$. The proof of Theorem~\ref{weak*-closed G inv} further requires Lemma~\ref{main lemma}, which we prove in Section~\ref{proof main lemma}.

\begin{lemma}\label{main lemma}
Let $Y\subset L^\infty(\mu)$ be a $G$-invariant space. Then for $g\in L^\infty(\mu)$, we have that $g\notin\overline{Y}^{2}$ whenever $g\notin\overline{Y}^{*}$.
\end{lemma}

\begin{remark}\label{main remark}
From Remark~\ref{weak* and weak Lp closure containment} and Lemma~\ref{main lemma}, for any $G$-invariant space $Y\subset L^\infty(\mu)$, $$\overline{Y}^{*}=\overline{Y}^{2}\cap L^\infty(\mu).$$
\end{remark}

\begin{remark}\label{main remark E}
Remark~\ref{main remark} and Remark~\ref{E Omega L2 pi i} give a description of the sets $\overline{E}_\Omega^{*}$: $$\overline{E}_\Omega^{*}=\overline{E}_\Omega^{2}\cap L^\infty(\mu)=\{f\in L^\infty(\mu):\pi_{i}f=0\text{ when }i\notin\Omega\}.$$
\end{remark}

\begin{proof}[\bf Proof of Theorem~\ref{weak*-closed G inv}.]
Let $Y\subset L^\infty(\mu)$ be a weak*-closed $G$-invariant space. Then $$Y=\overline{Y}^{*}=\overline{Y}^{2}\cap L^\infty(\mu)$$ from Remark~\ref{main remark}. Since $Y$ is $G$-invariant, so is $\overline{Y}^{2}$ from Corollary~\ref{Lp and C G-inv closure}. By Theorem~\ref{main C and Lp result}, $$\overline{Y}^{2}=\overline{E}_{\Omega'}^{2},$$ where $\Omega'=\{i\in I:\pi_{i}\overline{Y}^{2}\neq 0\}$.

We define $\Omega=\{i\in I:\pi_{i}Y\neq 0\}$. Then, Remark~\ref{main remark E} yields $$\overline{E}_\Omega^{2}\cap L^\infty(\mu)=\overline{E}_\Omega^{*}.$$ We have $\Omega=\Omega'$ by the continuity of each $\pi_{i}$, and thus $\overline{E}_\Omega^{2}=\overline{E}_{\Omega'}^{2}$.
\end{proof}


\section{Proof of Lemma~\ref{main lemma}.}\label{proof main lemma}

In this section, we prove Lemma~\ref{main lemma}, which we note is an analogue to Lemma~5.4 of \cite{hokampconti}, as well as a generalization of Lemma~4.2 from \cite{hokamp2021certain}.

\begin{lemma}\label{cont map}
Let $g\in L^\infty(\mu)$. Then the map $\phi:G\to L^\infty(\mu)$ given by $\phi(\alpha)=g\circ\varphi_\alpha$ is weak*-continuous.
\end{lemma}

\begin{proof}
To show that $\phi$ is weak*-continuous, we verify each $\Lambda_h\circ\phi$ is continuous, where $\Lambda_h$ is the map $L^\infty(\mu)\to\C$ given by integration against the function $h\in L^1(\mu)$.

Observe the map $\Lambda_h\circ\phi$ is given by $$(\Lambda_h\circ\phi)(\alpha)=\int_X(g\circ\varphi_\alpha)\cdot h\,d\mu.$$ From Lemma~6.2 of \cite{hokampconti}, the map $\alpha\mapsto h\circ\varphi_\alpha$ is continuous from $G$ into $L^1(\mu)$. Thus, the map $\alpha\mapsto g\cdot(h\circ\varphi_{\alpha^{-1}})$ is continuous. We apply Remark~\ref{switcher-mu} to get that $$\alpha\mapsto\int_Xg\cdot(h\circ\varphi_{\alpha^{-1}})\,d\mu=\int_X(g\circ\varphi_{\alpha})\cdot h\,d\mu$$ is continuous, as desired.
\end{proof}

\begin{proof}[\bf Proof of Lemma~\ref{main lemma}.]
Suppose $g\in L^\infty(\mu)$ and $g\notin\overline{Y}^{*}$. Then there exists a weak*-continuous linear functional $\Gamma$ on $L^\infty(\mu)$ such that $\Gamma f=0$ for $f\in Y$, and $\Gamma g=1$, due to the Hahn-Banach theorem (Theorem~3.5 \cite{RFA}). Since each weak*-continuous linear functional on $L^\infty(\mu)$ is induced by an element of $L^1(\mu)$, there exists $h\in L^1(\mu)$ such that $\Gamma F=\int_X Fh\,d\mu$ for $F\in L^\infty(\mu)$.

From Lemma~\ref{cont map}, there exists a neighborhood $N$ of the identity in $G$ such that $$\text{Re}\int_X (g\circ\varphi_\alpha)\cdot h\,d\mu>\frac{1}{2}$$ for $\alpha\in N$. We choose a continuous map $\psi:G\to [0,\infty)$ such that $\int\psi\,dm=1$ and the support of $\psi$ is contained in $N$ (recall $m$ denotes the Haar measure on $G$).

We now define a map $\Lambda$ on $L^\infty(\mu)$ by $$\Lambda F=\int_X h(x)\int_G \psi(\alpha)\cdot F(\alpha x)\,dm(\alpha)\,d\mu(x),\text{ for }F\in L^\infty(\mu).$$

We fix $F\in L^\infty(\mu)$ and $x\in X$ and define the map $\mathscr{F}_x:G\to\C$ by $\alpha\mapsto F(\alpha x)$. Since $$\int_G|\mathscr{F}_x|^2\,dm=\int_G |F(\alpha x)|^2\,dm(\alpha)=\int_X |F|^2\,d\mu=\n{F}_2^2<\infty,$$ we get $\mathscr{F}_x\in L^2(G)$, and since $\psi$ is continuous, we have that $\psi\in L^2(G)$. Further, 
$$\Big|\int_G \psi \mathscr{F}_x\,dm\Big|\leq\Big(\int_G|\psi|^2\,dm\Big)^\frac{1}{2}\Big(\int_G|\mathscr{F}_x|^2\,dm\Big)^\frac{1}{2}=\n{\psi}_2\cdot\n{F}_2,$$ so that for $F\in L^\infty(\mu)$, $$|\Lambda F|\leq\n{\psi}_2\cdot\n{F}_2\int_X|h|\,d\mu=\n{\psi}_2\cdot\n{F}_2\cdot\n{h}_1.$$

The linearity of $\Lambda$ on $L^\infty(\mu)$ is clear. Thus, $\Lambda$ defines an $L^2$-continuous linear functional on $L^\infty(\mu)$, and hence extends to an $L^2$-continuous linear functional $\Lambda_1$ on $L^2(\mu)$ by the Hahn-Banach theorem (Theorem~3.6 \cite{RFA}). By interchanging the integrals in the definition of $\Lambda$, we see that $\Lambda_1$ annihilates $Y$, since $Y$ is $G$-invariant. Further, $$\text{Re }\Lambda_1 g=\int_G\psi(\alpha)\Big(\text{Re}\int_X g(\alpha x)\cdot h(x)\,d\mu(x)\Big)\,dm(\alpha)>\int_N\psi(\alpha)\cdot\frac{1}{2}\,dm(\alpha)=\frac{1}{2}.$$ We conclude that $g\notin\overline{Y}^{2}$.
\end{proof}

\section{Future Questions}

\begin{itemize}
    \item[(1)] Does a $G$-collection exist for all groups $G$ acting continuously and transitively on $X$? What conditions might exist on $G$ or $X$ that yield a collection lacking ($*$)?
    \item[(2)] Under what conditions can the restrictions on $X$, $G$, and the action of $G$ on $X$ be loosened? Can the compactness of $X$ and $G$ be substituted with local compactness? Can the continuity of the action be substituted with separate continuity?
    \item[(3)] Suppose $H$ is a subgroup of $G$ and $\mathscr{H}$ is a collection of closed $H$-minimal spaces satisfying the same conditions as $\mathscr{G}$. What is the relationship between $\mathscr{H}$ and $\mathscr{G}$? What if $\mathscr{H}$ and $\mathscr{G}$ lack $(*)$? The uniqueness of $\mu$ shows that $H$ does not induce a new $H$-invariant measure on $X$. Further, $G$-invariance implies $H$-invariance (of a space).
    
    We note that (3) is prompted from the study of $\M$-invariant and $\U$-invariant spaces of continuous functions on the unit sphere of $\mathbb{C}^n$ from \cite{NR}, in which it is shown that there are infinitely many $\U$-invariant spaces and only six $\M$-invariant spaces. These six $\M$-invariant spaces are found by combining the $\U$-minimal spaces in a specific way (see Lemma~13.1.2 of \cite{RFT}), and we are curious if this method can be generalized.
    \item[(4)] Under what conditions can a $G$-collection characterize the closed $G$-invariant \textit{algebras} of continuous functions? We note that the case for the unitary group acting on the unit sphere of $\mathbb{C}^n$ is discussed in \cite{RUA} and is also summarized in \cite{RFT}.
\end{itemize}

\section{Data Availability Statement}

Data sharing not applicable to this article as no datasets were generated or analysed during the current study.

\bibliographystyle{unsrt}
\bibliography{bibliography}

\begin{thebibliography}{10}

\bibitem{hokamp2021certain}
Samuel~A. Hokamp.
\newblock {Certain invariant spaces of bounded measurable functions on a
  sphere}.
\newblock {\em Positivity}, 25(5):2081--2098, Nov 2021.

\bibitem{NR}
Alexander Nagel and Walter Rudin.
\newblock Moebius-invariant function spaces on balls and spheres.
\newblock {\em Duke Math. J.}, 43(4):841--865, 1976.

\bibitem{hokampconti}
Samuel~A. Hokamp.
\newblock Spaces of continuous and measurable functions invariant under a group
  action.
\newblock {\em https://arxiv.org/abs/2110.12060}, October 2021.

\bibitem{weil1940}
Andr\'e. Weil.
\newblock {L'int\'egration dans les groupes topologiques et ses applications}.
\newblock {\em Actual. Sci. Ind.}, no. 869, 1940.

\bibitem{joys}
Joe Diestel and Angela Spalsbury.
\newblock {\em {The Joys of Haar Measure}}.
\newblock American Mathematical Society, April 2014.

\bibitem{Izzo2010}
Alexander~J. Izzo.
\newblock {Uniform Algebras Invariant under Transitive Group Actions}.
\newblock {\em Indiana Univ. Math. J.}, 59(2):417--426, 2010.

\bibitem{RFT}
Walter Rudin.
\newblock {\em Function theory in the unit ball of {$\Bbb C^n$}}.
\newblock Classics in Mathematics. Springer-Verlag, Berlin, 2008.
\newblock Reprint of the 1980 edition.

\bibitem{peterweyl}
F.~Peter and H.~Weyl.
\newblock {Die Vollst{\ifmmode\ddot{a}\else\"{a}\fi}ndigkeit der primitiven
  Darstellungen einer geschlossenen kontinuierlichen Gruppe}.
\newblock {\em Math. Ann.}, 97(1):737--755, December 1927.

\bibitem{RFA}
Walter Rudin.
\newblock {\em Functional Analysis}.
\newblock McGraw-Hill, 1973.

\bibitem{RUA}
Walter Rudin.
\newblock Unitarily invariant algebras of continuous functions on spheres.
\newblock {\em Houston J. Math.}, 5(2):253--265, 1979.

\end{thebibliography}

\end{document}